\documentclass[reqno,12pt]{amsart}
\usepackage{amsfonts,amsmath,amsthm,amssymb,amscd}
\usepackage[all]{xy}
\usepackage[vcentermath]{youngtab}
\usepackage[dvipdfm]{color}
\usepackage[dvipdfm,bookmarks=true,bookmarksnumbered=true,bookmarkstype=toc]{hyperref}

\newcommand{\version}{Ver.~0.0}
\newcommand{\setversion}[1]{\renewcommand{\version}{Ver.~{#1}}}
\setversion{0.1 [Thu Apr 19 12:36:51 JST 2007]}
\setversion{0.2 [Thu Nov 22 11:45:39 JST 2007]}
\setversion{0.3 [2008/06/20 23:03]}
\setversion{0.4 [Thu Dec 25 17:05:05 JST 2008]}
\setversion{0.8 [2009/09/10 15:10]}
\setversion{1.0 [2009/09/10 17:02]}
\setversion{1.1 [2010/11/28 12:14]}
\setversion{1.2 [2010/11/28 12:14]}

\title [Asymptotic Cone for Symmetric Pair]
{Asymptotic cone of semisimple orbits for symmetric pairs}

\author{Kyo Nishiyama}
\address{
Department of Physics and Mathematics\\
Aoyama Gakuin University\\
Fuchinobe 5-10-1, Sagamihara 252-5258, Japan}
\email{kyo@gem.aoyama.ac.jp}

%22E46, 11F27

\thanks{Supported by JSPS Grant-in-Aid for Scientific Research \# 21340006.}

\theoremstyle{plain}
\newtheorem{theorem}{Theorem}
\newtheorem{proposition}[theorem]{Proposition}
\newtheorem{corollary}[theorem]{Corollary}
\newtheorem{lemma}[theorem]{Lemma}

\newtheorem{problem}[theorem]{\upshape Problem}
\theoremstyle{definition}
\newtheorem{definition}[theorem]{Definition}

\theoremstyle{remark}
\newtheorem{remark}[theorem]{\upshape Remark}

\numberwithin{equation}{section}
\numberwithin{theorem}{section}
\newcommand{\Z}{\mathbb{Z}}

\newcommand{\R}{\mathbb{R}}

\newcommand{\C}{\mathbb{C}}

\newcommand{\bbP}{\mathbb{P}}

\newcommand{\lie}[1]{\mathfrak{#1}}
\newcommand{\lier}[1]{\mathfrak{#1}_{\R}}

\newcounter{thmenum}
\newenvironment{thmenumerate}{%
\smallskip
\begin{list}{$(\thethmenum)$}{%
\usecounter{thmenum}
\setlength{\labelsep}{.5em}
\setlength{\labelwidth}{-7pt}
\setlength{\topsep}{0pt}
\setlength{\partopsep}{0pt}
\setlength{\parsep}{0pt}
\setlength{\leftmargin}{3pt}
\setlength{\rightmargin}{0pt}
\setlength{\itemindent}{\leftmargin}
\setlength{\itemsep}{0pt}
}}
{\end{list}\smallskip}

\newcounter{thmenumroman}
\renewcommand{\thethmenumroman}{\roman{thmenumroman}}

\newcommand{\mycomment}[1]{} % Nothing to do

\newlength{\lengthcup}
\newcommand{\mathoperator}[1]{\qopname\relax o{#1}}

\newcommand{\grade}{\qopname\relax o{gr}}

\newcommand{\vspan}{\qopname\relax o{span}}

\newcommand{\Lie}{\qopname\relax o{Lie}}

\newcommand{\Ind}{\qopname\relax o{Ind}}

\newcommand{\Ad}{\qopname\relax o{Ad}}
\newcommand{\ad}{\mathop{\mathrm{ad}}\nolimits{}}
\newcommand{\codim}{\qopname\relax o{codim}}

\newcommand{\iunit}{\mbox{\small$\sqrt{-1\,}\,$}}

\newcommand{\transpose}[1]{{}^t{#1}}

\newcommand{\AssVar}{\mathop\mathcal{AV}\nolimits{}}
\newcommand{\AssCycle}{\mathop\mathcal{AC}\nolimits{}}
\newcommand{\closure}[1]{\overline{#1}}

\newcommand{\trivial}{\mathbf{1}}

\newcommand{\conjugate}[1]{\overline{\rule{0pt}{1.2ex} #1}}

\newcommand{\irreps}[1]{\mathrm{Irr}(#1)}
\newcommand{\restrict}{\big|}
\newcommand{\partition}{\mathcal{P}}

\newcommand{\orbit}{\mathbb{O}}  % 
\newcommand{\calorbit}{\mathcal{O}}  % 
\newcommand{\GL}{\mathrm{GL}}

\newcommand{\U}{\mathrm{U}}

\newcommand{\Mat}{\mathrm{M}}

\newcommand{\nullcone}{\lie{N}}
\newcommand{\nullconeof}[1]{\nullcone(#1)}

\newcommand{\nilpotents}{\mathcal{N}}

\newcommand{\composit}{\odot}

\makeatletter
\newcommand{\adots}{\mathinner {\mkern 1mu\raise \p@ \vbox 
{\kern 7\p@ \hbox {.}}\mkern 2mu \raise 4\p@ \hbox {.}\mkern 2mu\raise 7\p@ \hbox {.}\mkern 1mu}}
\makeatother

\newcommand{\tensor}{\otimes}

\newcommand{\idealof}[1]{\mathbb{I}({#1})}

\newcommand{\projof}[1]{\mathbb{P}({#1})}
\newcommand{\proj}{\mathbb{P}}

\newcommand{\Acone}[1]{\mathfrak{C}^{\bbP}({#1})}
\newcommand{\affAcone}[1]{\mathfrak{C}({#1})}

\newcommand{\injection}{\hookrightarrow}
\newcommand{\surjection}{\twoheadrightarrow}

\newcommand{\llg}{\lie{g}}
\newcommand{\llk}{\lie{k}}
\newcommand{\lls}{\lie{s}}
\newcommand{\llp}{\lie{p}}

\newcommand{\llu}{\lie{u}}

\newcommand{\tds}{\mathfrak{sl}_2}

\newcommand{\GR}{G_{\R}}
\newcommand{\KR}{K_{\R}}
\newcommand{\MMR}{M_{\R}}
\newcommand{\PR}{P_{\R}}
\newcommand{\LR}{L_{\R}}
\newcommand{\NR}{N_{\R}}
\newcommand{\AR}{A_{\R}}
\newcommand{\gR}{\lie{g}_{\R}}
\newcommand{\sR}{\lie{s}_{\R}}

\newcommand{\Korbit}{\orbit^{K}}
\newcommand{\GRorbit}{\orbit^{\GR}}
\newcommand{\Gorbit}{\orbit^{G}}

\newcommand{\flagvariety}{\mathfrak{B}}

\newcommand{\forbit}{\mathcal{O}}

\newcommand{\cspan}[1]{\vspan_{\C}\{ #1 \}}

\begin{document}

\begin{abstract}
Let $ G $ be a reductive algebraic group over $ \C $ and denote its Lie algebra by $ \lie{g} $.  
Let $ \orbit_h $ be a closed $ G $-orbit through a semisimple element $ h \in \lie{g} $.  
By a result of Borho and Kraft \cite{Borho.Kraft.1979}, it is known that 
the asymptotic cone of 
the orbit $ \orbit_h $ is the closure of a Richardson nilpotent orbit corresponding to a parabolic subgroup 
whose Levi component is the centralizer $ Z_G(h) $ in $ G $.  
In this paper, we prove an analogue on a semisimple orbit for a symmetric pair.  

More precisely, 
let $ \theta $ be an involution of $ G $, and $ K = G^{\theta} $ a fixed point subgroup of $ \theta $.  
Then we have a Cartan decomposition $ \lie{g} = \lie{k} + \lie{s} $ of the Lie algebra $ \lie{g} = \Lie(G) $ 
which is the eigenspace decomposition of $ \theta $ on $ \lie{g} $.  
Let $ \{ x, h, y \} $ be a normal $ \tds $ triple, where $ x, y \in \lie{s} $ is nilpotent, and $ h \in \lie{k} $ semisimple.  
In addition, we assume $ \conjugate{x} = y $, where $ \conjugate{x} $ denotes the complex conjugation which commutes with $ \theta $.  
Then $ a = \iunit (x - y) $ is a semisimple element in $ \lie{s} $, and we can consider 
a semisimple orbit $ \Ad (K)\, a $ in $ \lie{s} $, which is closed.  
Our main result asserts that 
the asymptotic cone of $ \Ad (K) \, a $ in $ \lie{s} $ coincides with $ \closure{\Ad (G) \, x \cap \lie{s}} $, 
if $ x $ is even nilpotent.
\end{abstract}

\maketitle

%\tableofcontents
%\newpage

\section*{Introduction}

Let $ G $ be a connected reductive algebraic group over $ \C $ and denote its Lie algebra by $ \lie{g} $.  
Let $ h \in \lie{g} $ be a semisimple element and denote by $ \orbit_h $ the adjoint $ G $-orbit through $ h $.  
It is a closed affine subvariety in $ \lie{g} $.  
With this semisimple orbit, we can associate two objects.  

One object is a nilpotent orbit called a Richardson orbit.  
To be more precise, let us consider the centralizer $ L := Z_G(h) $ of $ h $.  
Then, there is a parabolic subgroup $ P $ whose Levi component is $ L $.  
Let us denote a Levi decomposition of the Lie algebra $ \lie{p} $ by 
$ \lie{l} + \lie{u} $, where $ \lie{u} $ denotes the nilpotent radical of $ \lie{p} $.  
Then $ \Ad(G) \lie{u} $ is the closure of a single nilpotent orbit $ \calorbit $, which 
is called the Richardson orbit associated with $ P $.  
The Richardson orbit $ \calorbit $ in fact does not depend on the choice of the parabolic $ P $, 
and it is determined by $ h $.

The other object, which we consider, is the asymptotic cone 
$ \affAcone{\orbit_h} $ of $ \orbit_h $, which indicates the asymptotic direction in which 
the variety $ \orbit_h $ spreads out.  See \S \ref{section:asymptotic.cone} for precise definition.  

In \cite{Borho.Kraft.1979}, Borho and Kraft studied Dixmier sheets, and in the course of their study they proved 
the following theorem.  

\begin{theorem}[Borho-Kraft]
For a semisimple orbit $ \orbit_h $, the asymptotic cone $ \affAcone{\orbit_h} $ 
coincides with the closure of the Richardson nilpotent orbit $ \closure{\calorbit} $ above.
\end{theorem}

This can be interpreted as a generalization of Kostant's theorem, 
which asserts that the nilpotent variety $ \nilpotents(\lie{g}) $ is 
a deformation of the regular semisimple orbits (\cite{Kostant.1963}).  
Note that $ \nilpotents(\lie{g}) $ is the closure of a principal nilpotent orbit, 
which is a Richardson orbit associated with a Borel subgroup.  
In this case, the ``deformation" amounts to taking an asymptotic cone of regular semisimple orbits. 

In this paper, we prove an analogous theorem for a semisimple orbit for a symmetric pair.  

Let us explain it more precisely.  
Let $ \theta $ be an involution of $ G $, and $ K = G^{\theta} $ a fixed point subgroup of $ \theta $.  
Then we have a Cartan decomposition $ \lie{g} = \lie{k} + \lie{s} $ of the Lie algebra $ \lie{g} = \Lie(G) $ 
which is the eigenspace decomposition of $ \theta $ on $ \lie{g} $.  
We pick a nilpotent element $ x $ in $ \lie{s} $, and 
consider a normal $ \tds $ triple $ \{ x, h, y \} $, where $ x, y \in \lie{s} $ is nilpotent, and $ h \in \lie{k} $ semisimple.  
In addition, we can assume $ \conjugate{x} = y $ without loss of generality, 
where $ \conjugate{x} $ denotes the complex conjugation which commutes with $ \theta $.  
Then $ a = \iunit (x - y) $ is a semisimple element in $ \lier{s} $, and we can consider 
a semisimple orbit $ \Korbit_a = \Ad (K)\, a $ in $ \lie{s} $, which is closed.  

Our main result asserts that, 
if $ x $ is even nilpotent, 
the asymptotic cone of $ \Korbit_a $ in $ \lie{s} $ coincides with $ \closure{\Gorbit_x \cap \lie{s}} $, 
where $ \Gorbit_x = \Ad(G) x $ is a nilpotent $ G $-orbit through $ x $.  
In fact, the intersection $ \Gorbit_x \cap \lie{s} $ breaks up into several nilpotent $ K $-orbits, 
\begin{equation*}
\Gorbit_x \cap \lie{s} = \bigcup\nolimits_{i = 0}^{\ell} \Korbit_{x_i} ,
\end{equation*}
each of which is a Lagrangian subvariety of $ \Gorbit_x $.  
So we can state our main theorem as

\begin{theorem}
Suppose $ x \in \lie{s} $ is an even nilpotent element, and 
construct a semisimple element $ a \in \lier{s} $ as explained above.  
Then the asymptotic cone of the semisimple orbit $ \Korbit_a $ in $ \lie{s} $ 
is given by 
\begin{equation*}
\affAcone{\Korbit_a} = \closure{\Gorbit_x \cap \lie{s}} = \bigcup\nolimits_{i = 0}^{\ell} \closure{\Korbit_{x_i}} .
\end{equation*}
\end{theorem}

Note that the asymptotic cone is no longer irreducible in the case of symmetric pair.  
This reflects the reducibility of the nilpotent variety for symmetric pairs 
as pointed out by \cite{Kostant.Rallis.1971}.  
Our theorem can be seen as a generalization of Kostant-Rallis's theorem.

\medskip

From the semisimple element $ a \in \lier{s} $, 
we can construct a real parabolic subgroup $ \PR $ in a standard way (see \S \ref{section:main.theorem}).  
The asymptotic cone above is the associated variety of a degenerate principal series representation 
$ \Ind_{\PR}^{\GR} \chi $ induced from a character $ \chi $ of $ \PR $.  
It seems that the irreducible components $ \Korbit_{x_i} $ of $ \affAcone{\Korbit_a} $ play 
a important role in the theory of degenerate principal series representations.  
We discuss what we can expect for this, using an example in the case of $ \GR = U(n, n) $ in \S \ref{section:example.unn}.

\section{Asymptotic Cone}
\label{section:asymptotic.cone}

Let $ V = \C^N $ be a vector space.  
For a subvariety $ X \subset V $, we define the asymptotic cone of $ X $, denoted by $ \Acone{X} \subset \proj(V) $, as follows.  
We extend $ V $ by the one-dimensional vector space, and denote it by $ \tilde{V} = V \oplus \C $.  
We consider the projective space $ \proj(\tilde{V}) $.  
Then there is a natural open embedding $ \iota : V \hookrightarrow \proj(\tilde{V}) $ defined by $ \iota(v) = [ v \oplus 1 ] $, 
where $ [w] $ denotes the image of $ w \in \tilde{V} \setminus \{ 0 \} $ in $ \proj(\tilde{V}) $ under the natural projection.  
On the other hand, there is a closed embedding $ \kappa : \proj(V) \hookrightarrow \proj(\tilde{V}) $ 
which send $ [u] \in \proj(V) $ to $ \kappa([u]) := [ u \oplus 0 ] \in \projof{\tilde{V}} $.  
Thus we have a disjoint decomposition $ \projof{\tilde{V}} = \iota(V) \sqcup \kappa(\projof{V}) $.  
In the following, we identify $ \projof{V} $ with $ \kappa(\projof{V}) $ and 
consider it as a closed subvariety of $ \projof{\tilde{V}} $.

\begin{definition}
Let $ X $ be a subvariety of $ V $ of positive dimension. 
We define the asymptotic cone of $ X $ by $ \Acone{X} := \closure{\iota(X)} \cap \projof{V} $, 
where $ \projof{V} $ is identified with $ \kappa(\projof{V}) \subset \projof{\tilde{V}} $.  
Then $ \Acone{X} \subset \projof{V} $ is a projective variety of the same dimension as $ X $.  
The affine cone in $ V $ associated to $ \Acone{X} $ is denoted by $ \affAcone{X} $, and 
we call it the \emph{affine asymptotic cone}, while $ \Acone{X} $ is called the \emph{projective asymptotic cone}.

If $ X $ is $ 0 $-dimensional, i.e., if it consists of a finite set of points, 
we put $ \Acone{X} = \emptyset $ and $ \affAcone{X} = \{ 0 \} $.
\end{definition}

The asymptotic cone was introduced by W.~Borho and H.~Kraft (\cite{Borho.Kraft.1979}) 
to study Dixmier sheets of the adjoint representation of a reductive algebraic group.  
We refer the readers to \cite{Borho.Kraft.1979} for the details of their properties.  
Here in this section we only recall some properties of asymptotic cones without proof.

Let $ I $ be an ideal of the polynomial ring $ \C[V] $.  
For $ f \in I $, let $ \grade f $ be the homogeneous part of the maximal degree.  
We define $ \grade I = ( \grade f \mid f \in I ) $, 
the homogeneous ideal generated by $ \grade f \; ( f \in I ) $.  

Let $ \idealof{X} $ be the annihilator ideal of $ X $.  
Then the annihilator ideal of the asymptotic cone is given by 
$ \idealof{\affAcone{X}} = \sqrt{\grade \idealof{X}} $.  
Thus the regular function ring $ \C[\affAcone{X}] $ is isomorphic to $ \C[V] / \sqrt{\grade \idealof{X}} $, 
which is equal to the homogeneous function ring of $ \Acone{X} $. 

Let $ G $ be a connected reductive algebraic group over $ \C $ 
which acts linearly on $ V $ and assume that $ X $ is stable under $ G $.  
Then the ring of regular functions $ \C[X] $ has a natural $ G $-module structure.  
The asymptotic cone $ \Acone{X} $ as well as $ \affAcone{X} $ is also a $ G $-variety, 
and we have a $ G $-action on the regular function ring $ \C[\affAcone{X}] $ in particular.

\begin{lemma}
Let $ X $ be a closed affine variety in $ V $ which is stable under the action of $ G $, 
and $ I = \idealof{X} $ an annihilator ideal of $ X $.  
Then $ \C[X] \simeq \C[V] / I $ is isomorphic to $ \C[V] / \grade I $ as a $ G $-module.  
Since $ \C[\affAcone{X}] \simeq \C[V] / \sqrt{I} $, 
we have a surjective $ G $-module morphism $ \C[X] \surjection \C[\affAcone{X}] $.
\end{lemma}

Let $ \nullconeof{V} := \{ v \in V \mid f(v) = 0 \; (f \in \C[V]_+^G) \} $ be the null fiber.  
It is the zero locus of homogeneous $ G $-invariants of positive degree.  

\begin{proposition}
Let $ \orbit $ be a $ G $-orbit in $ V $.  
Then the affine asymptotic cone $ \affAcone{\orbit} $ is a $ G $-stable subvariety of $ \nullconeof{V} $, 
which is equidimensional and $ \dim \affAcone{\orbit} = \dim \orbit $.
\end{proposition}

Let $ \lie{g} $ be a Lie algebra on which $ G $ acts by the adjoint action.  
Then the null fiber $ \nullconeof{\lie{g}} $ is called the nilpotent variety, which consists of 
all the nilpotent elements in $ \lie{g} $.  
It is well known that $ \nullconeof{\lie{g}} $ contains only a finite number of $ G $-orbits.  

\begin{corollary}\ 
For $ x \in \lie{g} $, let $ \orbit_x = \Ad(G) \, x $ be the adjoint orbit through $ x $.  
Then the affine asymptotic cone $ \affAcone{\orbit_x} $ is a finite union of 
the closure of nilpotent orbits, whose dimension is equal to $ \dim \orbit_x $.
\end{corollary}

In the following, we will denote the adjoint action simply by 
$ g x = \Ad(g) x $ for $ g \in G , \, x \in \llg $.

\section{Richardson Orbit}
\label{section:Richardson.orbit}

Let $ h \in \llg $ be a semisimple element, and put $ L := Z_G(h) $ 
the centralizer of $ h $ in $ G $.  
There is a parabolic subgroup $ P $ with a Levi decomposition 
$ P = L U $, where $ U $ is the unipotent radical.  
Then 
$ \llp = \lie{l} \oplus \llu $ 
is a Levi decomposition of the corresponding Lie algebra.

\begin{definition}
Let $ \llu $ be the nilpotent radical of a parabolic subalgebra $ \llp $.  
Then adjoint translate $ G \llu = \{ \Ad(g) u \mid g \in G , u \in \llu \} $ of $ \llu $ is 
the closure of a single nilpotent orbit 
$ \closure{\orbit_x} \; (x : \text{nilpotent element}) $.  
We call $ \orbit_x $ the \emph{Richardson orbit} for the parabolic $ P $, and 
$ x $ a \emph{Richardson element}.  
We often assume $ x $ to be taken from $ \llu $.
\end{definition}

Let us consider a partial flag variety 
$ \flagvariety_P := G / P $ of all parabolics conjugate to $ \llp $, 
and denote by 
$ T^{\ast} \flagvariety_P $ the cotangent bundle over $ \flagvariety_P $.  
Then there is a $ G $-equivariant map $ \mu $ called the \emph{moment map} defined as follows.
\begin{equation*}
\mu : T^{\ast} \flagvariety_P \simeq G \times_P \llu \ni ( g, z ) \to \Ad (g) z \in \llg 
\end{equation*}

The following proposition is well known.  See \cite{Jantzen.2004} and references therein.

\begin{proposition}
Assume that $ x $ is a Richardson element for $ P $ and that $ Z_G(x) = Z_P(x) $ holds.
\begin{thmenumerate}
\item
The moment map 
$ \mu : T^{\ast} \flagvariety_P \to \closure{\orbit_x} $ is a resolution of singularities of $ \closure{\orbit_x} $.
\item
The fiber of $ \orbit_x $ is $ \mu^{-1}(\orbit_x) = G [e, x] $ and 
$ \mu : G [e, x] \xrightarrow{\sim} \orbit_x $ is an isomorphism.
\item
The moment map $ \mu $ induces a $ G $-equivariant isomorphism 
$ \C[ G \times_P \llu ] = \C[ G \times \llu]^P \simeq \C[ \orbit_x ] $.
In addition, 
if $ \closure{\orbit_x} $ is normal, then $ \C[ \closure{\orbit_x} ] = \C[\orbit_x] $ holds.
\end{thmenumerate}
\end{proposition}

If a reductive group $ K $ acts on a variety $ \mathfrak{X} $, we get a decomposition of the regular function ring 
as a $ K $-module, 
\begin{equation}
\label{eqn:notation.multiplicity}
\C[\mathfrak{X}] \simeq \textstyle\bigoplus\nolimits_{\tau \in \irreps{K}} m_{\tau}(\mathfrak{X}) \, \tau 
\qquad (\text{as a $ K $-module}), 
\end{equation}
where $ m_{\tau}(\mathfrak{X}) $ denotes the multiplicity.

\begin{theorem}[Borho-Kraft]
\label{theorem:Borho.Kraft}
Let $ h \in \lie{g} $ be a semisimple element 
and define the parabolic subgroup $ P $ and the Richardson orbit $ \orbit_x $ 
as above.  
Then the asymptotic cone of the semisimple orbit $ \orbit_h $ is equal to the Richardson orbit : 
$ \affAcone{\orbit_h} = \closure{\orbit_x} $.  
In addition, 
if $ Z_G(x) $ is connected and $ \closure{\orbit_x} $ is normal, we have 
\begin{equation*}
\C[ \orbit_h ] \simeq \Ind_L^G \trivial_L \simeq \C[ \orbit_x ] = \C[ \closure{\orbit_x} ] 
= \C[ \affAcone{\orbit_h} ] 
\qquad
(\text{as $ G $-modules})
\end{equation*}
i.e., $ m_{\tau}( \orbit_h ) = m_{\tau}( \orbit_x ) = m_{\tau}( \affAcone{\orbit_h} ) = \dim \tau^L $ 
$ (\forall \tau \in \irreps{G}) $.
\end{theorem}

Up to this point, we started with a semisimple element, but now we investigate in other ways.  
So take a nilpotent element $ x \in \llg $, and choose 
an \emph{$ \tds $ triple} $ \{ x , h, y \} $,  
where $ h $ is semisimple; $ x, y $ are nilpotent; and they satisfy the commutation relations 
\begin{equation*}
[ h, x ] = 2 x , \quad [ h, y ] = - 2 y , \quad [ x , y ] = h .
\end{equation*}
Thus $ \llg $ is a representation space of $ \tds = \cspan{ x, h, y } $.  
Therefore the eigenvalues of $ \ad h $ are integers and we get a $ \Z $-grading of $ \llg $ 
induced by the action of $ \ad h $.
\begin{equation}
\label{eqn:z.grading.by.adh}
\lie{g} = \textstyle\bigoplus\nolimits_{k \in \Z} \lie{g}_k \qquad
\lie{g}_k := \{ X \in \lie{g} \mid \ad (h) X = k X \} 
\end{equation}

\begin{definition}
If $ \lie{g}_1 = \{ 0 \} $, $ x $ is called an \emph{even} nilpotent element. 
Note that $ \lie{g}_1 = \{ 0 \} $ if and only if $ \lie{g}_k = \{ 0 \} \; (\forall k \text{ : odd} ) $.
\end{definition}

We put 
$ \lie{p} = \textstyle\bigoplus\nolimits_{k \geq 0} \lie{g}_k = \lie{l} \oplus \llu $, 
where 
$ \lie{l} = \lie{g}_0 $ and 
$ \llu = \textstyle\bigoplus\nolimits_{k > 0} \lie{g}_k $.  
Then $ \llp $ is a parabolic subalgebra 
and, if $ x $ is even nilpotent, then $ \orbit_x $ is a Richardson orbit for $ P = N_G(\llp) $.  
Even nilpotent elements have good properties (see \cite{Jantzen.2004} for example).

\begin{proposition}
Assume $ x $ is even nilpotent, 
then $ Z_G(x) = Z_P(x) $ holds.  
Hence the moment map $ \mu : T^{\ast} \flagvariety_P \to \closure{\orbit_x} $ is a resolution of singularities, 
and we have an isomorphism of regular function rings 
$ \C[ T^{\ast} \flagvariety_P ] \simeq \C[ \orbit_x ] $.  

Moreover, if $ \closure{\orbit_x} $ is normal, then 
$ \C[ \closure{\orbit_x} ] \simeq \C[ \orbit_x ] \simeq \C[ T^{\ast} \flagvariety_P ] $.
\end{proposition}

\begin{corollary}
Let $ \{ x , h, y \} $ be an $ \tds $ triple with $ x $ even nilpotent and 
assume that $ \closure{\orbit_x} $ is normal.  
Then the asymptotic cone of a semisimple element $ h $ is equal to the closure of the nilpotent orbit through $ x $.
\begin{equation*}
\affAcone{\orbit_h} = \closure{\orbit_x} 
\end{equation*}
Moreover, there is an isomorphism 
$ \C[ \affAcone{\orbit_h} ] \simeq \C[ T^{\ast} \flagvariety_P ] $.
\end{corollary}

\section{Richardson orbit for symmetric pair}
\label{section:Richardson.orbit.for.symmetric.pair}

Let $ \GR $ be a reductive Lie group, which is a real form of 
a connected complex algebraic group $ G $.  
We fix a Cartan involution $ \theta $.  
Then the fixed point subgroup of $ \theta $ is a maximal compact subgroup $ \KR = \GR^{\theta} $.  
We extend $ \theta $ to $ G $ holomorphically, and 
put $ K = G^{\theta} $, which is a complexification of $ \KR $.  
We mainly consider a symmetric pair $ ( G , K ) $ in the following.  

Let $ \lie{g} $ be the Lie algebra of $ G $, and 
$ \lie{g} = \lie{k} \oplus \lie{s} $ a (complexified) Cartan decomposition, 
where $ \lie{k} $ is the Lie algebra of $ K $ and $ \lie{s} $ is the $ (-1) $-eigenspace of the differential of $ \theta $.  

Take a $ \theta $-stable parabolic subalgebra $ \lie{p} $ of $ \lie{g} $.  
We denote by $ P $ the corresponding parabolic subgroup of $ G $, 
and put $ \flagvariety_P = G / P $, the partial flag variety.  
Then $ \flagvariety_P $ can be considered as the totality of 
the parabolic subalgebras of $ \lie{g} $ which is conjugate to $ \lie{p} $ by the adjoint action of $ G $.  
The $ K $-orbit of the $ \theta $-stable parabolic $ \lie{p} $ is a closed orbit in $ \flagvariety_P $.   
Conversely, if there is a $ \theta $-stable parabolic, 
then any closed $ K $-orbit in $ \flagvariety_P $ arises as a $ K $-conjugacy class of $ \theta $-stable parabolic subalgebras.  

Let $ \forbit $ denote a closed $ K $-orbit in $ \flagvariety_P $ generated by $ \lie{p} $.  
Then the conormal bundle $ T_{\forbit}^{\ast} \flagvariety_P $ over $ \forbit $ can be described as follows.  

Since $ \lie{p} $ is $ \theta $-stable, 
$ \lie{q} = \lie{p} \cap \lie{k} $ is a parabolic subalgebra in $ \lie{k} $.  
Let $ Q $ be the corresponding parabolic subgroup of $ K $.  
If $ \lie{p} = \lie{l} \oplus \lie{u} $ is a $ \theta $-stable Levi decomposition, 
$ \lie{q} = \lie{l}(\lie{k}) \oplus \lie{u}(\lie{k}) $ with 
$ \lie{l}(\lie{k}) = \lie{l} \cap \lie{k} $ and $ \lie{u}(\lie{k}) = \lie{u} \cap \lie{k} $ gives 
a Levi decomposition of $ \lie{q} $.  
Also we put $ \lie{u}(\lie{s}) = \lie{u} \cap \lie{s} $.
Then $ \lie{u}(\lie{s}) $ is $ Q $-stable, and we have
\begin{equation*}
T_{\forbit}^{\ast} \flagvariety_P \simeq K \times_Q \lie{u}(\lie{s}) = (K \times \lie{u}(\lie{s})) / Q
\end{equation*}
where the action of $ Q $ on $ K \times \lie{u}(\lie{s}) $ is given by $ q ( k, x) = (k q^{-1} , \Ad (q) x) $ for $ q \in Q , k \in K , x \in \lie{u}(\lie{s}) $.  
We denote the class of $ ( k, x ) \in K \times \lie{u}(\lie{s}) $ in $ K \times_Q \lie{u}(\lie{s}) $ by $ [ k, x ] $.  
Then a map 
\begin{equation*}
\mu : T_{\forbit}^{\ast} \flagvariety_P \simeq K \times_Q \lie{u}(\lie{s}) \to \lie{s} , \qquad  \mu( [k, x] ) = \Ad (k) x 
\end{equation*}
is well-defined, and called the \emph{moment map}.  
For any $ K $-orbit $ \forbit $ in $ \flagvariety_P $, 
the moment map image of the conormal bundle $ T_{\forbit}^{\ast} \flagvariety_P $ is the closure of 
a single nilpotent $ K $-orbit $ \Korbit $ in $ \lie{s} $.  
The following definition is due to P.~Trapa \cite{Trapa.2005} 
(see also \cite{Trapa.2007}).

\begin{definition}
Let $ \lie{p} $ be a $ \theta $-stable parabolic subalgebra and 
$ \forbit $ a closed $ K $-orbit in $ \flagvariety_P $ through $ \lie{p} $.  
If a nilpotent $ K $-orbit $ \Korbit \subset \lie{s} $ is dense in 
the moment map image of $ T_{\forbit}^{\ast} \flagvariety_P $, 
it is called a \emph{Richardson orbit for the symmetric pair} $ G/K $ associated to $ \lie{p} $.
\end{definition}

The following is a representation theoretic characterization of Richardson orbits.

\begin{theorem}
A nilpotent $ K $-orbit $ \Korbit \subset \lie{s} $ is a Richardson orbit for the symmetric pair 
if and only if its closure is the associated variety of a derived functor module 
$ A_{\lie{p}} $ with the trivial infinitesimal character for a certain $ \theta $-stable parabolic subalgebra $ \lie{p} $.
\end{theorem}

\section{Asymptotic cone for symmetric pair}
\label{section:main.theorem}

Let $ x \in \lie{s} $ be a nilpotent element.  
Then we can choose $ y \in \lie{s} $ and $ h \in \lie{k} $ such that 
$ \{ x, h, y \} $ forms a normal $ \tds $ triple, where $ x, y $ are nilpotent, and $ h $ semisimple 
(see \cite[\S 9.4]{Collingwood.McGovern.1993} for example).  
In addition, after suitable conjugation by $ K $, 
we can assume $ \conjugate{x} = y $, where $ \conjugate{x} $ denotes the complex conjugation 
with respect to $ \gR $.  We call a normal $ \tds $ triple with this property a \emph{KS triple}.  
Then 
\begin{equation*}
a = \iunit (x - y) \in \sR 
\end{equation*}
is a semisimple element in $ \sR $.   
Also we put 
\begin{equation*}
e = \frac{1}{2} ( x + y + \iunit h ) , \qquad
f = \frac{1}{2} ( x + y - \iunit h ) = - \theta(e) .
\end{equation*}
Then $ e $ and $ f $ are nilpotent elements belonging to the real form $ \gR $, 
and $ \{ e , a , f \} $ is a standard $ \tds $ triple in $ \gR $.  
We call it a \emph{Cayley triple}.  
Every standard $ \tds $ triple is $ \GR $-conjugate to a Cayley triple.

The following theorem is well known.

\begin{theorem}[Sekiguchi \cite{Sekiguchi.1987}, Vergne \cite{Vergne.1995}] 
Nilpotent orbits $ \Korbit_x = \Ad (K) \, x $ and $ \GRorbit_e = \Ad (\GR) \, e $ are $ \KR $-equivariantly diffeomorphic, 
and moreover they generate the same nilpotent $ G $-orbit: 
$ \Ad (G) x = \Ad (G) e $.  
This correspondence gives a bijection between the set of non-zero nilpotent $ K $-orbits in $ \lie{s} $ and 
that of non-zero nilpotent $ \GR $-orbits in $ \gR $.  
\end{theorem}

See \cite[Theorem 9.5.1 \& Remark 9.5.2]{Collingwood.McGovern.1993} and \cite{Barbasch.Sepanski.1998} for further properties.  

\medskip

Let us denote $ \Gorbit_x = \Ad (G) x $.  
Then the intersection $ \Gorbit_x \cap \lie{s} $ breaks up into several nilpotent $ K $-orbits 
$ \textstyle\bigcup\nolimits_{i = 0}^{\ell} \Korbit_{x_i} $ where $ x = x_0 $.  
It is well known that each $ \Korbit_{x_i} $ is a Lagrangian subvariety 
for the canonical symplectic structure on 
$ \Gorbit_x $, and 
consequently they all have the same dimension $ \frac{1}{2} \dim \Gorbit_x $ 
(see \cite[Corollary 5.20]{Vogan.1991} for example).  
We also consider a complex semisimple orbit $ \Korbit_a := \Ad (K)\, a \subset \lie{s} $, which is closed.   
Note that $ a $ and $ h $ generate the same $ G $-orbit, $ \Gorbit_a = \Ad (G) \, a = \Gorbit_h $.

Let us consider $ \ad h $-eigenspace decomposition 
\begin{math}
\lie{g} = \textstyle\bigoplus\nolimits_{k \in \Z} \lie{g}_k 
\end{math}
as in Equation \eqref{eqn:z.grading.by.adh}.
We put 
\begin{equation}
\label{eqn:parabolic.from.z.grading}
\lie{p} = \textstyle\bigoplus\nolimits_{k \geq 0} \lie{g}_k = \lie{l} \oplus \lie{u} , 
\qquad 
\text{ where } \quad
\lie{l} = \lie{g}_0 , \; 
\lie{u} = \textstyle\bigoplus\nolimits_{k > 0} \lie{g}_k .
\end{equation}
Then $ \lie{p} $ is a $ \theta $-stable parabolic subalgebra, and 
$ \lie{q} = \lie{p} \cap \lie{k} $ is a parabolic in $ \lie{k} $.  
We denote $ P $ and $ Q $ the parabolic subgroups of $ G $ and $ K $ respectively 
corresponding to $ \lie{p} $ and $ \lie{q} $.  
We follow the notation in \S \ref{section:Richardson.orbit.for.symmetric.pair}.

\begin{theorem}
\label{theorem:asymptotic.cone.of.semisimple.orbit}
Assume that $ x \in \lie{s} $ is an even nilpotent element, and 
let $ \{ x , h , y \} $ be a normal $ \tds $ triple.  
After conjugation by $ K $, we can assume $ \{ x, h, y \} $ is a KS triple.  
Put $ a = \iunit ( x - y ) \in \sR $.   
Then the asymptotic cone of $ \Korbit_a $ is equal to 
\begin{equation}
\label{eqn:asymptotic.cone.of.semisimple.orbit}
\affAcone{\Korbit_a} = \closure{\Gorbit_x \cap \lie{s}} 
%% This is wrong: = \closure{\Gorbit_x} \cap \lie{s} 
= \textstyle\bigcup\nolimits_{i = 0}^{\ell} \closure{\Korbit_{x_i}} \; , 
\end{equation}
where $ \{ x = x_0 , x_1 , \dots x_{\ell} \} $ is a complete set of representatives of the $ K $-orbits in 
$ \Gorbit_x \cap \lie{s} $, 
and $ \{ \Korbit_{x_i} \; ( 0 \leq i \leq \ell ) \} $ are Richardson orbits for a symmetric pair $ G/K $.
\end{theorem}

\begin{proof}
Since $ x $ is even nilpotent by assumption, 
the $ K $-orbit $ \Korbit_x $ is a Richardson orbit corresponding to the $ \theta $-stable parabolic $ \lie{p} $ in 
\eqref{eqn:parabolic.from.z.grading}.  
See \cite{Noel.2005} for details.  
For $ 1 \leq i \leq \ell $, because $ x_i $ is a $ G $-translate of $ x $, they are all even nilpotent.  
Thus the same reasoning can be applied to the orbits $ \Korbit_{x_i} $ which tells us that they are all Richardson.  

Now let us consider $ a = \iunit ( x - y ) $.  Then we calculate
\begin{equation*}
\exp ( t \ad h ) a 
= \iunit ( e^{2 t} x - e^{ - 2 t} y ) 
= \iunit e^{2 t} ( x - e^{ - 4 t} y ) .
\end{equation*}
Therefore we get in $ \projof{ \llg \oplus \C } $, 
\begin{equation*}
[ \exp ( t \ad h ) a \oplus 1 ] 
=  [ (x - e^{ - 4 t} y) \oplus (- \iunit e^{-2 t} ) ]
\to [ x \oplus 0 ] \in \kappa( \projof{\llg} )
\qquad 
( t \to \infty ) .
\end{equation*}
This proves that 
$ x \in \affAcone{\Korbit_a} $ and hence 
$ \closure{\orbit_x} \subset \affAcone{\Korbit_a} $ because $ \affAcone{\Korbit_a} $ 
is a $ K $-invariant closed set.
By the same reason, we get 
$ \closure{\orbit_{x_i}} \subset \affAcone{\Korbit_{a_i}} $, 
where $ a_i $ is defined similarly as $ a $ by using $ x_i $ instead of $ x $.  

%%One can prove that $ a_i $'s are all $ K $-conjugate to $ a $.  
%%We will give the proof of this fact due to T.~Matsuki in the Appendix.  

The semisimple elements $ a_i $'s are in fact all conjugate to $ a $ by the adjoint action of $ K $.  
This follows from the fact that representatives of the little Weyl group (the Weyl group of the restricted root system) 
can be chosen from the elements in $ K $ (\cite[Corollary 6.55]{Knapp.2002}).

Thus we have proved that the right hand side is contained in the asymptotic cone $ \affAcone{\Korbit_a} $.

On the other hand, from Theorem \ref{theorem:Borho.Kraft}, we clearly have
\begin{equation*}
\affAcone{ \Korbit_a } \subset \affAcone{ \Gorbit_a } \cap \lie{s} \subset \closure{\Gorbit_x} \cap \lie{s} .
\end{equation*}
Thus we get 
\begin{equation*}
\closure{\Gorbit_x \cap \lie{s}} \subset \affAcone{ \Korbit_a } \subset \closure{\Gorbit_x} \cap \lie{s} .
\end{equation*}
Note that $ \closure{\Gorbit_x \cap \lie{s}} $ is a union of all irreducible components of $ \closure{\Gorbit_x} \cap \lie{s} $ 
of maximal dimension $ \frac{1}{2} \dim \Gorbit_x $ (cf.~Remark \ref{remark:closure.of.intersection.of.Gorbit.and.s}(1) below).  
Since $ \affAcone{ \Korbit_a } $ is equi-dimensional, it must coincide with $ \closure{\Gorbit_x \cap \lie{s}} $.
\end{proof}

\begin{remark}
\label{remark:closure.of.intersection.of.Gorbit.and.s}
\begin{thmenumerate}
\item
The inclusion $ \closure{\Gorbit_x \cap \lie{s}} \subset \closure{\Gorbit_x} \cap \lie{s} $ 
might be strict.  For example, consider a symmetric pair $ ( G, K ) = ( \GL_{2n} , \GL_n \times \GL_n ) $ 
which is associated to $ \U(n, n) $.  
Take the nilpotent $ G $-orbit $ \Gorbit $ of Jordan type $ [3\cdot 1^{2 n - 3}] $.  
Then $ \closure{\Gorbit_x \cap \lie{s}} $ consists of the $ K $-orbits whose 
signed Young diagrams are 
\begin{align*}
&
[(+-+)\cdot(+)^{n - 2}\cdot(-)^{n -1}], \quad
[(-+-)\cdot(+)^{n - 1}\cdot(-)^{n -2}], \quad \\
&
[(+-)\cdot(-+)\cdot(+)^{n - 2}\cdot(-)^{n -2}], \quad \\
&
[(+-)\cdot(+)^{n - 1}\cdot(-)^{n -1}], \quad
[(-+)\cdot(+)^{n - 1}\cdot(-)^{n -1}], \quad \\
&
[(+)^{n}\cdot(-)^{n}], 
\end{align*}
while the $ K $-orbits 
\begin{math}
[(+-)^2\cdot(+)^{n - 2}\cdot(-)^{n -2}] \text{ and } 
[(-+)^2\cdot(+)^{n - 2}\cdot(-)^{n -2}] 
\end{math}
are not contained in the closure but contained in $ \closure{\Gorbit_x} \cap \lie{s} $.  
See the Hasse diagram of the closure relation below.
\\
\begin{table}[h!]
{
\scriptsize
\newcommand{\smallvdots}{\makebox[0pt]{\raisebox{-.19em}{$\cdot$}}\makebox[0pt]{\raisebox{.1em}{$\cdot$}}\makebox[0pt]{\raisebox{.39em}{$\cdot$}}}
\hfil
\begin{minipage}{.4\textwidth}
\hfil{\normalsize$ \affAcone{\Korbit_a} = \closure{\Gorbit_x \cap \lls} $}\hfil
\\[7pt]
\xymatrix @R-25pt @C-15pt {
& {\young({+}-+,+,\smallvdots,-)} \ar[dr] & & \ar[dl] {\young({-}+-,+,\smallvdots,-)} & \\
& & \ar[dl] {\young({+}-,-+,+,\smallvdots,-)} \ar[dr] & & \\
& {\young({+}-,+,\smallvdots,-)} \ar[dr] & & \ar[dl] {\young({-}+,+,\smallvdots,-)} & \\
& & {\young({+},\smallvdots,-)} & & 
}
\end{minipage}
\begin{minipage}{.4\textwidth}
\hfil{\normalsize$ \closure{\Gorbit_x} \cap \lls $}\hfil
\\[7pt]
\xymatrix @R-25pt @C-18pt {
& {\young({+}-+,+,\smallvdots,-)} \ar[dr] & & \ar[dl] {\young({-}+-,+,\smallvdots,-)} & \\
{\young({+}-,+-,+,\smallvdots,-)} \ar[dr] & & \ar[dl] {\young({+}-,-+,+,\smallvdots,-)} \ar[dr] & & \ar[dl] {\young({-}+,-+,+,\smallvdots,-)} \\
& {\young({+}-,+,\smallvdots,-)} \ar[dr] & & \ar[dl] {\young({-}+,+,\smallvdots,-)} & \\
& & {\young({+},\smallvdots,-)} & & 
}
\end{minipage}
\hfil
}
\end{table}

\item
The collection of $ \{ \Korbit_{x_i} \; ( 0 \leq i \leq \ell ) \} $ is a set of Richardson orbits 
which are the moment map image of the conormal bundle of closed $ K $-orbits in 
the fixed partial flag variety $ \flagvariety_P $ through $ \theta $-stable parabolics 
(not necessarily all of them).  
Let us denote a closed $ K $-orbit in $ \flagvariety_P $ by $ \forbit_i $ which corresponds to 
the Richardson orbit $ \Korbit_{x_i} $.  
If $ K_{x_i} $ is connected, 
the moment map $ \mu_i : T_{\forbit_i}^{\ast} \flagvariety_P \to \closure{\Korbit_{x_i}} $ is a resolution of the singularities 
(see Proposition 5.9 and \S 8.8 of \cite{Jantzen.2004}).  
\end{thmenumerate}
\end{remark}

Since $ a \in \sR $ is a real hyperbolic element, it naturally defines a real parabolic subalgebra $ \lier{p} $, 
which is the non-negative part of the $ \Z $-grading similar to \eqref{eqn:parabolic.from.z.grading}
 with respect to $ \ad a $ instead of $ \ad h $.    
Let us denote by $ \PR $ the corresponding real parabolic subgroup of $ \GR $.  
A parabolically induced representation from a character $ \chi $ of $ \PR $ is called 
the \emph{degenerate principal series representation}, 
which is denoted by $ I_{\PR}(\chi) = \Ind_{\PR}^{\GR} \chi $.

\begin{corollary}
\label{cor:asymptotic.cone.as.associated.variety.of.deg.PS}
We assume $ x \in \lie{s} $ is even nilpotent and use the setting of Theorem \ref{theorem:asymptotic.cone.of.semisimple.orbit}.
Let $ I_{\PR}(\chi) $ be a degenerate principal series representation of $ \GR $, where $ \PR $ is obtained from $ a \in \sR $ as above.  
Then the associated variety of $ I_{\PR}(\chi) $ is equal to 
the asymptotic cone $ \affAcone{\Korbit_a} $ $($see Equation \eqref{eqn:asymptotic.cone.of.semisimple.orbit}$)$.  
\end{corollary}

\begin{proof}
It is known that the $ G $-hull of the associated variety $ \AssVar(I_{\PR}(\chi)) $ is 
the closure of the Richardson $ G $-orbit associated to $ P $.  
Thus, by Theorem \ref{theorem:asymptotic.cone.of.semisimple.orbit}, we have $ \AssVar(I_{\PR}(\chi)) \subset \affAcone{\Korbit_a} $.  
Note that the function ring 
$ \C[ \AssVar(I_{\PR}(\chi)) ] $ is asymptotically 
isomorphic to the space of $ \KR $-finite vectors in $ I_{\PR}(\chi) $ as $ \KR $-modules.
If $ \chi $ is trivial, we have
\begin{equation*}
I_{\PR}(\trivial) \restrict_{\KR} \simeq \Ind_{\MMR}^{\KR} \trivial \simeq \C[ \Korbit_a ] , 
\qquad 
\MMR = Z_{\KR}(a) .
\end{equation*}
Therefore, asymptotically $ \C[ \AssVar(I_{\PR}(\trivial)) ] $ and $ \C[\affAcone{\Korbit_a}] $ are equal.  
So they must coincide with each other.
\end{proof}

\begin{remark}
The wave front set of $ I_{\PR}(\chi) $ is known by the results in \cite{Barbasch.Bozicevic.1999} (see also \cite{Barbasch.2000}).  
Therefore, using Schmid-Vilonen's theorem \cite{Schmid.Vilonen.2000}, 
we basically know the associated variety of $ I_{\PR}(\chi) $.  
Here, in the corollary above, 
the emphasis is on the coincidence with the asymptotic cone.
\end{remark}

The conclusion of 
Corollary \ref{cor:asymptotic.cone.as.associated.variety.of.deg.PS} 
does not contain the even nilpotent element $ x $ explicitly.  
In fact, it is plausible to believe the conclusion is always true.

\begin{problem}
Let $ a \in \sR $ be a hyperbolic semisimple element and define the parabolic $ \lier{p} $ as above.  
Does the associated variety of the degenerate principal series $ I_{\PR}(\chi) $ coincide with 
the asymptotic cone $ \affAcone{\Korbit_a} $?  
\end{problem}

\begin{remark}
\begin{thmenumerate}
\item
For a general $ a \in \sR $, it is no longer true that the asymptotic cone $ \affAcone{\Korbit_a} $ is equal to 
the intersection of the closure of the Richardson orbit and $ \lie{s} $.  
For this, we refer to an example in \cite[Example 3.8]{Matumoto.Trapa.2007}. 
\item
There is a formula for the asymptotic $ K $-support by T.~Kobayashi, which is very close to the above problem.  
His formula (\cite[Theorem 6.4.3]{Kobayashi.2005}) implies  
\begin{equation*}
\mathoperator{AS}_K( I_{\PR}(\chi) \restrict_{\KR} ) = C^+ \cap \iunit \Ad^{\ast}(\KR) ( \lier{m} )^{\bot} ,
\end{equation*}
where $ C^+ $ denotes the closed Weyl chamber inside $ \iunit \lier{t}^{\ast} $.  
However, up to now, we do not know the exact relation of the above formula to our problem.
\end{thmenumerate}
\end{remark}

\begin{corollary}
Suppose that $ x \in \lie{s} $ is even nilpotent which satisfies 
\begin{thmenumerate}
\item
the fixed point subgroup $ K_x $ is connected, 
\item
$ \closure{\Korbit_x} $ is normal,
\item
$ \codim \partial \Korbit_x \geq 2 $, where $ \partial \Korbit_x = \closure{\Korbit_x} \setminus \Korbit_x $ is 
the boundary of $ \Korbit_x $.
\end{thmenumerate}
Then the intersection $ \Gorbit_x \cap \lie{s} = \Korbit_x $ consists of a single $ K $-orbit.  
If we take a KS triple $ \{ x , h, y \} $ as above, 
the asymptotic cone of the semisimple orbit $ \Korbit_a $ $ ( a = \iunit ( x - y ) ) $ is given by 
$ \affAcone{\Korbit_a} = \closure{\Korbit_x} $.  
In this case, we have isomorphisms of algebra 
\begin{equation*}
\C[ T_{\forbit}^{\ast} \flagvariety_P] \simeq \C[ \Korbit_x ] \simeq \C[ \closure{\Korbit_x} ] ,
\end{equation*}
and, as $ K $-modules, they are isomorphic to $ \C[ \Korbit_a ] $.
\end{corollary}

\begin{proof}
We use the following lemma.  
Let us recall the notation $ m_{\tau}(\mathfrak{X}) $ for the multiplicity defined in \eqref{eqn:notation.multiplicity}.

\begin{lemma}
The following inequality holds.
\begin{equation*}
m_{\tau}( \Korbit_a ) \geq 
m_{\tau}(\affAcone{\Korbit_a}) \geq 
m_{\tau}(\closure{\Korbit_{x_i}}) \qquad ( \tau \in \irreps{K}) .
\end{equation*}
\end{lemma}

\begin{proof}
Let us denote the annihilator ideal of $ \Korbit_a $ by $ I = \idealof{\Korbit_a} \subset \C[\lie{s}] $.  
Then we have $ \C[\Korbit_a] \simeq \C[\lie{s}] / \grade I $ as $ K $-modules.  
Moreover, there is a surjective algebra morphism 
$ \C[\lie{s}] / \grade I \surjection \C[\lie{s}] / \sqrt{\grade I} = \C[\affAcone{\Korbit_a}] $.  
Since this morphism is $ K $-equivariant, 
we have the following inequality
\begin{equation*}
m_{\tau}( \Korbit_a ) \geq m_{\tau}(\affAcone{\Korbit_a}) \qquad ( \tau \in \irreps{K}) .
\end{equation*}
Since $ \closure{\Korbit_{x_i}} $ in 
Theorem \ref{theorem:asymptotic.cone.of.semisimple.orbit} 
is an irreducible component of $ \affAcone{\Korbit_a} $, we also have an inequality 
\begin{math}
m_{\tau}(\affAcone{\Korbit_a}) \geq m_{\tau}(\closure{\Korbit_{x_i}}) .
\end{math}
\end{proof}

Let us return to the proof of the corollary.

By Theorem \ref{theorem:asymptotic.cone.of.semisimple.orbit}, 
we know $ \affAcone{\Korbit_a} $ is the union of $ \closure{\orbit}_{x_i} $'s.  
By Corollary \ref{cor:asymptotic.cone.as.associated.variety.of.deg.PS}, $ \affAcone{\Korbit_a} $ is an associated variety of a 
degenerate principal series $ I_{\PR}(\chi) $.  
For a generic parameter $ \chi $, the degenerate principal series representation is irreducible.  
So by Vogan's theorem (\cite[Theorem 4.6]{Vogan.1991}), 
if there are more than two irreducible components of the associated variety, they must 
have a codimension one orbit in its boundary.  
But by the assumption, there is no such orbit, hence it must be irreducible.

The normality and the codimension-two condition imply the isomorphism 
$ \C[\closure{\orbit_x}] \xrightarrow{\sim} \C[\orbit_x] $.  
Since $ K_x $ is connected the moment map $ \mu : T_{\forbit}^{\ast} \flagvariety_P \to \closure{\Korbit_x} $ is 
a resolution.  By \cite[Proposition 8.9]{Jantzen.2004}, we get $ \C[ T_{\forbit}^{\ast} \flagvariety_P] \simeq \C[ \Korbit_x ] $.  
\end{proof}

\section{Example: Siegel parabolics}
\label{section:example.unn}

Let $ \GR = \U(n, n) $ and $ \KR = \U(n) \times \U(n) $ a maximal compact subgroup.  
Then $ G = \GL_{2n}(\C) $ is the complexification of $ \GR $ 
and $ K = \GL_n(\C) \times \GL_n(\C) $ is block diagonally embedded into $ G $.  
$ ( G, K ) $ is a symmetric pair.  
The Cartan decomposition $ \llg = \llk \oplus \lls $ is given as follows.
\begin{equation*}
\llk = \{ \begin{pmatrix} A & 0 \\ 0 & D \end{pmatrix} \mid A , D \in \Mat_n(\C) \} , \quad
\lls = \{ \begin{pmatrix} 0 & B \\ C & 0 \end{pmatrix} \mid B , C \in \Mat_n(\C) \}
\end{equation*}
Let us consider a nilpotent element 
\begin{equation*}
x = \begin{pmatrix} 0 & 1_n \\ 0 & 0 \end{pmatrix} \in \lls .
\end{equation*}
If we put $ y = \transpose{x} $ and $ h = [ x, y ] $, then 
$ \{ x, h, y \} $ constitute a KS triple.  
Note that, in this case, 
the complex conjugation $ \sigma $ with respect to the real form $ \gR $ is given by 
\begin{equation*}
\sigma(X) = - I_{n, n} \transpose{\conjugate{X}} I_{n,n} \quad ( X \in \llg ) , \qquad 
I_{n, n} = \begin{pmatrix} 1_n & 0 \\ 0 & - 1_n \end{pmatrix} .
\end{equation*}
We can check $ \sigma(x) = \transpose{x} = y $ directly.  

The nilpotent element $ x $ generates a nilpotent $ G $-orbit $ \Gorbit_x $ which has Jordan type $ [2^n] $.  
Consequently $ x $ is even nilpotent.  
There are $ ( n + 1 ) $ nilpotent $ K $-orbits in $ \Gorbit_x \cap \lls $, which are 
$ \Korbit_{p,q} = [(+-)^p (-+)^q] \; ( p, q \geq 0 , p + q = n ) $ in the notation of signed Young diagram 
(see \cite{Collingwood.McGovern.1993}, for example).  

Put $ a = \iunit ( x - y ) \in \sR $.  
Theorem \ref{theorem:asymptotic.cone.of.semisimple.orbit} tells us that 
\begin{equation*}
\affAcone{ \Korbit_a } = \textstyle\bigcup_{p + q = n} \closure{\Korbit_{p,q}} .
\end{equation*}
Let us interpret the meaning of this identity in terms of the representation theory of $ \GR $.  

First, let us see the function ring $ \C[ \Korbit_a ] $.  
Put $ M = Z_K(a) $, the stabilizer of $ a $ in $ K $.  
Then clearly $ M = \Delta \GL_n(\C) $, 
the diagonal embedding of $ \GL_n(\C) $ into $ K = \GL_n(\C) \times \GL_n(\C) $.
Thus we have
\begin{equation}
\label{eqn:C.Korbit.a.as.induced.representation}
\C[ \Korbit_a ] = \C[ K/M ] = \C[ K ]^M \simeq \Ind_M^K \trivial_M , 
\end{equation}
where the last isomorphism is an isomorphism as $ K $-modules, 
and $ \trivial_M $ denotes the trivial representation of $ M $.
Thus we have 
\begin{equation}
\label{eqn:decomposition.of.C.Korbit.a}
\C[ \Korbit_a ] \simeq \textstyle\bigoplus_{ \rho \in \irreps{\GL_n}} \rho \tensor \rho^{\ast} 
\qquad
(\text{as a $ K \simeq \GL_n \times \GL_n $-module}),
\end{equation}
which is a multiplicity free $ K $-module.  
This is isomorphic to $ \C[ \affAcone{ \Korbit_a } ] $ as a $ K $-module 
by \cite[Theorem 3.1]{NOZ.2006}.

On the other hand, by explicit calculation using the technique in \cite{Nishiyama.MA.2000} 
(also see \cite{Nishiyama.JMK.2004}), we have
\begin{equation*}
\C[ \closure{\Korbit_{p,q}} ] \simeq 
\textstyle\bigoplus\limits_{\alpha \in \partition_p , \, \beta \in \partition_q} 
\rho_{\alpha \composit \beta} \tensor \rho_{\alpha \composit \beta}^{\ast} .
\end{equation*}
However, we have the following

\begin{proposition}
For any $ p, q \geq 0 $ satisfying $ p + q = n $, there are isomorphisms of $ K $-modules
\begin{equation*}
\C[ \Korbit_{p,q} ] \simeq \C[ \affAcone{\Korbit_a} ] \simeq \C[ \Korbit_a ] , 
\end{equation*}
where the first isomorphism is also a morphism of algebras induced by the 
open embedding 
$ \Korbit_{p,q} \injection \affAcone{\Korbit_a} $. 
\end{proposition}

Let us denote $ \MMR = Z_{\KR}(a) = \Delta \U(n) $, and 
$ \LR = Z_{\GR}(a) \simeq \GL_n(\C) $.  
The semisimple element $ a $ naturally defines 
a maximal parabolic subgroup $ \PR = \LR \NR $.  
where $ \NR $ is a suitably chosen unipotent radical.  
Note that $ \AR = \exp \R a $ is contained in the center of $ \LR = \GL_n(\C) $ as 
the radial part of the complex torus.  
We consider a degenerate principal series representation 
induced from a one dimensional character of $ \PR $ (unnormalized induction)
\begin{equation*}
I(\nu) := \Ind_{\PR}^{\GR} ( |\det|^{\nu + 2n} \tensor \trivial_{\NR} ) , 
\qquad ( \nu \in \C ) , 
\end{equation*}
where $ \det $ is the determinant character of $ \LR = \GL_n(\C) $ and 
the induced character is trivial on $ \NR $.  
Then we have
\begin{equation*}
I(\nu) \restrict_{\KR} 
\simeq \Ind_{\MMR}^{\KR} \trivial_{\MMR} 
\simeq \textstyle\bigoplus_{ \rho \in \irreps{\U(n)}} \rho \tensor \rho^{\ast} .
\end{equation*}
Comparing this with 
\eqref{eqn:decomposition.of.C.Korbit.a} and 
\eqref{eqn:C.Korbit.a.as.induced.representation}, 
we conclude that $ \Korbit_a $ or $ \affAcone{ \Korbit_a } $ carries information of $ K $-types 
of degenerate principal series $ I(\nu) $.  

\begin{theorem}[Sahi, Lee, Johnson, Wallach, ...]
Assume that $ \nu \geq 0 $ is even.  
Then the degenerate principal series $ I(\nu) $ 
contains precisely $ ( n + 1 ) $ irreducible subrepreesntations 
$ \pi_{p,q}(\nu) \; ( p , q \geq 0 , \; p + q = n ) $, 
which are unitary.  
If $ \nu > 0 $, then these are only unitarizable irreducible constituents of $ I(\nu) $.  
\end{theorem}

\begin{remark}
$ I(\nu) $ is reducible if and only if $ \nu $ is an even integer.  
If $ \nu \geq 0 $ (and even), then the Hasse diagram of subquotients of $ I(\nu) $ is given below 
(see \cite[\S\S 7{\&}9]{Lee.1994} and also \cite{Sahi.1993}).  

\begin{figure}[htbp]
\hfil
\xymatrix @R-10pt @C-10pt {
& & & & \ar[dl] \otimes \ar[dr] & & & \\
& & & \ar[dl] \circ \ar[dr] & & \ar[dl] \circ \ar[dr] & & \\
& & \ar[dl] \circ \ar[dr]  & &  \ar[dl] \circ \ar[dr]  & &  \ar[dl] \circ \ar[dr] \\
& \ar[dl] \circ \ar[dr]  & &  \ar[dl] \circ \ar[dr]  & &  \ar[dl] \circ \ar[dr]  & &  \ar[dl] \circ \ar[dr] \\
\bullet & & \bullet & & \bullet & & \bullet & & \bullet &
}
\quad
\makebox[20pt]{\small
\begin{tabular}[t]{l}
\quad\\[50pt]
$ \bullet $ : unitary\\
$ \circ $ : non-unitary\\
{\tiny$ \otimes $} : finite dimensional\\
\qquad unitary iff $ \nu = 0 $
\end{tabular}
}
\hfil
\\[10pt]
$ n = 4 $ : Hasse diagram of submodules of $ I(\nu) \; ( \nu \in 2 \Z_{\geq 0} ) $\\
\end{figure}
\begin{figure}[htbp]
\hfil
{\small
\xymatrix @R-10pt @C-15pt {
& & & & \ar[dl] \Korbit_{0,0} \ar[dr] & & & \\
& & & \ar[dl] \Korbit_{0,1} \ar[dr] & & \ar[dl] \Korbit_{1,0} \ar[dr] & & \\
& & \ar[dl] \Korbit_{0,2} \ar[dr]  & &  \ar[dl] \Korbit_{1,1} \ar[dr]  & &  \ar[dl] \Korbit_{2,0} \ar[dr] \\
& \ar[dl] \Korbit_{0,3} \ar[dr]  & &  \ar[dl] \Korbit_{1,2} \ar[dr]  & &  \ar[dl] \Korbit_{2,1} \ar[dr]  & &  \ar[dl] \Korbit_{3,0}  \ar[dr]\\
\Korbit_{0,4} & &  \Korbit_{1,3} & & \Korbit_{2,2} & & \Korbit_{3,1} & & \Korbit_{4,0}
}
}
\hfil
\\[10pt]
Hasse diagram of associated varieties
\end{figure}

If $ \nu = 0 $, then $ I(\nu) $ contains the trivial representation.  
In general $ I(\nu) \; (\nu \geq 0) $ contains a finite dimensional representation 
as a unique irreducible subrepresentation.
\end{remark}

If $ \nu = - n $, then $ I(-n) $ is a direct sum of $ (n + 1) $ irreducible unitary representations 
$ \{ \pi_{p,q}(-n) \mid p + q = n \} $, which are derived functor modules $ A_{\llp_{p,q}} $ 
(see \cite{Matumoto.Trapa.2007}).  
The representations $ \pi_{p,q}(\nu) \;( p + q = n) $ are translation (or coherent continuation) of these derived functor modules.

\begin{corollary}
The associated variety of $ I(\nu) $ is equal to 
\begin{math}
\affAcone{ \Korbit_a } = \textstyle\bigcup_{p + q = n} \closure{\Korbit_{p,q}} .
\end{math}
The associated cycle of the largest constituents $ \pi_{p,q}(\nu) \; ( p + q = n ) $ is given by 
$ \AssCycle{\pi_{p,q}(\nu)} = [ \closure{\Korbit_{p,q}} ] $ with multiplicity one.
\end{corollary}

%%%%%%%%%%%%%%%%%%%%%%%%%%%%%%%%%%%%%%%%%%%%%%%%%%%%%%%%%%%%%%%%%%%%%%%%%%%%%%
%\include{ref_asymptotic_cone}
%%%%%%%%%%%%%%%%%%%%%%%%%%%%%%%%%%%%%%%%%%%%%%%%%%%%%%%%%%%%%%%%%%%%%%%%%%%%%%

%%\bibliographystyle{amsalpha}
%%\bibliography{bib_acone}

\providecommand{\bysame}{\leavevmode\hbox to3em{\hrulefill}\thinspace}
\renewcommand{\MR}[1]{}
%%\providecommand{\MR}{\relax\ifhmode\unskip\space\fi MR }
% \MRhref is called by the amsart/book/proc definition of \MR.
\providecommand{\MRhref}[2]{%
  \href{http://www.ams.org/mathscinet-getitem?mr=#1}{#2}
}
\providecommand{\href}[2]{#2}

\end{document}